\documentclass[12pt]{article}
\usepackage{amsfonts,amssymb,amsmath,titling,titlesec,cancel,amsthm,color, hyperref}
\usepackage{booktabs}
\usepackage{pifont}
\usepackage{xcolor}
\usepackage{tikz}
\usepackage{multicol}
\usepackage{caption,subcaption}
\usetikzlibrary{arrows,shapes,trees}
\usepackage{epstopdf}

\DeclareMathOperator*{\sat}{sat} 
\DeclareMathOperator*{\emin}{emin}
\DeclareMathOperator*{\ex}{ex}
\DeclareMathOperator*{\cd}{cd}

\titleformat{\section}{\large\bfseries}{\thesection.}{1em}{}

\theoremstyle{plain}
\newtheorem{theorem}{Theorem}[section]
\newtheorem{lemma}[theorem]{Lemma}
\newtheorem{proposition}[theorem]{Proposition}

\newtheorem{conjecture}[theorem]{Conjecture}

\theoremstyle{definition}
\newtheorem{definition}[theorem]{Definition}

\title{Saturation Number of Trees in the Hypercube}
\author{Kavish Gandhi and Chiheon Kim}
\date{\today}

\begin{document}

\maketitle
\begin{abstract} %done
A graph $H^{\prime}$ is $(H, G)$-saturated if it is $G$-free and the addition of any edge of $H$ not in $H^{\prime}$ creates a copy of $G$. The \emph{saturation number} $\sat(H, G)$ is the minimum number of edges in an $(H, G)$-saturated graph. We investigate bounds on the saturation number of trees $T$ in the $n$-dimensional hypercube $Q_n$. We first present a general lower bound on the saturation number based on the minimum degree of non-leaves. From there, we suggest two general methods for constructing $T$-saturated subgraphs of $Q_n$ and prove nontrivial upper bounds for specific types of trees, including paths, generalized stars, and certain caterpillars under a restriction on minimum degree with respect to diameter.
\end{abstract}

\section{Introduction}
\label{sec:intro}
In 1941, Paul Tur{\'a}n \cite{Turan} proved one of the first important results in extremal graph theory, explicitly determining the maximum number of edges in a $K_{r+1}$-free subgraph of $K_n$.  This result sparked the study of what is now known as the extremal number.  In particular, the \emph{extremal number} $\ex(H, G)$ is defined as the maximum number of edges in a subgraph of a host graph $H$ that does not contains some forbidden graph $G$.

The extremal number also has a natural opposite formulation.  In particular, consider the following definition, where $E(G)$ represents the set of edges in a graph $G$:
\begin{definition}
A subgraph $H^{\prime}$ of $H$ is \emph{$(H, G)$-saturated} if it is $G$-free, but the addition of any edge in $E(H) \ \backslash \ E(H^{\prime})$ to $H^{\prime}$ creates a copy of $G$.
\end{definition}

The extremal number defines the \emph{maximum} number of edges in such a saturated graph, but we can also ask the converse question: what is the \emph{minimum} number of edges in a saturated graph?  To this end, the \emph{saturation number} $\sat(H, G)$ is defined as the minimum number of edges in a $G$-saturated subgraph of $H$.

For both the extremal number and the saturation number, the most well-studied host graphs have been the complete graph $K_n$ and the complete bipartite graph $K_{m, n}$.  We, on the other hand, study subgraphs of the \emph{hypercube} $Q_n$, the regular graph with vertex set $\{0, 1\}^n$ and edge set consisting of all pairs of vertices differing in exactly one coordinate.  Erd{\"o}s \cite{Erdos} was among the first to study $\ex(Q_n, G)$, specifically trying to determine $\ex(Q_n, C_4)$.  This question is still open today; the best bounds, due to Brass et al.~\cite{Brass} and Baber \cite{Baber}, are $(n+\sqrt{n})2^{n-2} \le \ex(Q_n, C_4) \le (0.60318 + o(1))n2^{n-1}$.\footnote{For the sake of accuracy, note that this lower bound is only valid when $n$ is a power of $4$.} Erd{\"o}s also conjectured that, for all $k \ge 2$, $\ex(Q_n, C_{2k}) = o(e(Q_n))$, but was proven wrong by Chung \cite{Chung}, who showed that $\ex(Q_n, C_6) \ge \frac{1}{4}e(Q_n)$.\footnote{Note that Conder later improved this lower bound, showing that $\ex(Q_n, C_6) \ge \frac{1}{3}e(Q_n)$.} However, she also showed that $\ex(Q_n, C_{4t}) = o(e(Q_n))$ for $t\ge 2$.  These results were expanded by F{\"u}redi and {\"O}zkahya \cite{Furedi}, who proved that $\ex(Q_n, C_{4t+2}) = o(e(Q_n))$ for $t\ge 3$, a result later shown in a more general framework by Conlon \cite{Conlon}. The only remaining unresolved asymptotic case is $C_{10}$, which, despite some progress \cite{Alon}, remains open.

By and large, though $\ex(Q_n, G)$ and in particular $\ex(Q_n, C_{2k})$ have been well studied, $\sat(Q_n, G)$ has not.  The best early result was of Choi and Guan \cite{ChoiGuan}, who showed that 
\[\lim_{n \to \infty}\frac{\sat(Q_n, Q_2)}{e(Q_n)} \le \frac{1}{4}.\]
Recently, Johnson and Pinto \cite{JohnsonPinto} improved this result, showing that 
\[\lim_{n \to \infty}\frac{\sat(Q_n, Q_m)}{e(Q_n)} = 0.\]
Morrison, Noel, and Scott \cite{Morrison}, even more recently, improved these bounds further, showing that
\[(m-1+o(1))\cdot 2^n \le \sat(Q_n, Q_m) < (1+o(1))72m^22^n.\]
However, little work has been done in determining $\sat(Q_n, T)$ for trees $T$, and this is the main concern of this paper. We first present a lower bound based on the minimum degree of non-leaves of a graph, and an upper bound for trees decomposable into subtrees with smaller cubical dimension, obtained by considering disjoint subcubes. We then use a variation on the Hamming code to find improved bounds for specific trees whose minimum degree is large compared to their diameter. This construction appears very different from constructions using disjoint subcubes, and we suspect that it can be extended to obtain tight upper bounds for all trees. 

The outline of this paper is the following.  In Section $\ref{sec:definitions}$, we define terms used frequently in the paper. Then, in Section $\ref{sec:outline}$, we present our results and outline the major ideas used in their proofs.  From there, in Sections $\ref{sec:genresults}$, $\ref{sec:treesdisjoint}$, and $\ref{sec:treeshamming}$, we present the proofs of all of these results, along with useful lemmata.  Finally, in Section $\ref{sec:conclusion}$, we summarize our work and suggest future directions.

\vspace{-0.1in}
\section{Definitions}
\label{sec:definitions}
In this section, we provide definitions for some standard terms used frequently in our paper. We first define some specific types of trees.  Note that our definitions of certain standard trees may differ slightly from their traditional definitions, so we present these too.
\begin{definition}
A \emph{path} $P_k$ is a sequence of vertices $v_1 v_2 \ldots v_{k+1}$ connected consecutively by edges such that $v_i \ne v_j$ for $i \ne j$.  
\end{definition}
Note that $P_k$ is more traditionally defined as the path with $k$ vertices, but since in saturation we deal mostly with the number of edges, our definition is more suitable.

\begin{definition}
A \emph{star} $S_k$ is the complete bipartite graph $K_{1, k}$.  In other
words, $S_k$ has one central vertex and $k$ leaves
connected to this vertex.
\end{definition}

\begin{definition}
A \emph{generalized star} $GS_{k,m}$ consists of one central vertex and
$k$ disjoint paths of length $m$ emanating from this vertex.  We call each
disjoint path a \emph{leg} of $GS_{k, m}$.
\end{definition}

\begin{definition}
A \emph{caterpillar $S_{m_1 \times m_2 \times \cdots \times m_k}$} is the tree in which $k$ vertices with degree greater than $1$ form a central path, and in which these vertices have degree $m_1, m_2, \dotsc, m_k$ in the order they appear on the central path.
\end{definition}

Note that caterpillars are more traditionally defined as trees in which each vertex is either on or adjacent to some central path; our definition is more useful for our needs because it allows us to exactly specify the degree of each vertex along the central path.

Our next two definitions are for convenience, as they are frequently referenced in proofs.
\begin{definition}
The \emph{weight} of a vertex $v \in V(Q_n)$, denoted by $w(v)$, is the number of $1$'s in the binary representation of $v$.
\end{definition}
\begin{definition}
The \emph{Cartesian product} $G \ \Box \ H$ is the graph created by placing copies of $G$ at all of the vertices of $H$ and connecting corresponding vertices of adjacent $G$'s in $H$.  
\end{definition}

\noindent We specifically use the fact that $Q_{n} = Q_{k} \ \Box \ Q_{n-k}$ to create some of our saturated subgraphs.  

Our final definition is a special subset of the vertices of the hypercube that we use in Section $\ref{sec:treeshamming}$ as a building block for the construction of saturated subgraphs.  Note that our definition is very similar to that in \cite{JohnsonPinto}.

\begin{definition}
Given $n = 2^{i} - 1$, let $H$ be the $i$ by $n$ matrix whose columns are all nonzero vectors in $\mathbb{F}_2^{i}$.  The \emph{Hamming code}
in $Q_n$ is the nullspace of $H$ over $\mathbb{F}_2$.
\end{definition}
\noindent A Hamming code $C$ in $Q_n$ satisfies the following properties:
\begin{enumerate}
\item{$|C| = \frac{2^n}{n+1}$.}
\item{The distance between all pairs of vertices in $C$ is at least $3$.}
\item{$C$ dominates $Q_n$. Precisely, every vertex in $V(Q_n) \setminus C$ is adjacent to exactly one vertex in $C$.}
\end{enumerate}

It is also helpful sometimes to consider a dominating set of $Q_n$ consist of disjoint subcubes $Q_j$ for $j < n$. A set $C$ of vertices of $Q_n$ is called a\emph{perfect dominating set} of $Q_n$ if every vertex is either in $C$ or is adjacent to exactly one vertex in $C$. Note that the Hamming code in $Q_{2^{i}-1}$ is an example of such a perfect dominating set.  Weichsel \cite{Weichsel} found a similar dominating set of disjoint subcubes $Q_j$:

\begin{theorem}[\cite{Weichsel}]
\label{thm:hamming_cube}
Let $n = 2^t - 1 + m$ with $0 \leq m < 2^t$ and $r = n - (2^s - 1)$ for some fixed $s$, $0 \leq s \leq t$. Then there is a perfect dominating set $C$ of the cube $Q_n$ such that the induced subgraph of $Q_n$ on $C$ is the disjoint union of $\frac{2^{n - r}}{n - r + 1}$ copies of $Q_r$. 
\end{theorem}

%%%%%%%%%%%%%%%%%%%%%%%%%%%%%%%%%%%%%%%%%%%
%%%%%%%%%%%%%%%%%%%%%%%%%%%%%%%%%%%%%%%%%%%

\section{Outline of Results}
$\label{sec:outline}$
In this section, we state our major results, accompanied often by a sketch of the proof or the main idea behind the proof.  The first such result is a general lower bound on $\sat(Q_n, G)$.  To derive this bound, we use an argument based on the minimum degree of
$G$, finding a lower bound that is best for graphs with large minimum degree of non-leaves. Before stating the result, it is necessary
to first define $\emin(G)$ as the minimum value, over all pairs of adjacent vertices in $G$, of the maximum degree of two adjacent vertices.
\begin{theorem}
\label{thm:lowerbound}
Given a graph $G$ with $\emin(G)=\delta$, $\sat(Q_n, G) \ge (\delta - 1 + o(1)) \cdot 2^{n-2}$.  
\end{theorem}
The main idea behind this proof is that, for every non-edge in $Q_n$, one of the endpoints must have degree at least $\delta -1$ in our subgraph.  This idea, along with a few tweaks, yields the desired bound.  It is important to note here that this result only uses the fact that $Q_n$ is $n$-regular, so actually applies not only in $Q_n$ but in any $n$-regular host graph.

Our next result is a general lower bound on the saturation number of  trees.  Let, for a tree $T$, the \emph{cubical dimension} of $T$, denoted by $\cd(T)$, be the smallest positive integer for which $T$ can be embedded in $Q_{\cd(T)}$.  Then we have the following result:
\begin{theorem}
\label{thm:supahimportant}
Let 
\[k := \min\limits_{e \in T}\ \{ \max(\cd(T_1), \cd(T_2))\},\]
where $T_1$ and $T_2$ are the two connected components of $T\setminus e$.  Then, if $k<\cd(T)$, there exists a
$T$-saturated subgraph of $Q_n$ with $k \cdot 2^{n-1}$ edges.
\end{theorem}
The main idea of this proof is that we can place $T_1$ and $T_2$ in disjoint $k$-dimensional subcubes and, because of our condition on the cubical dimension, not have a copy of $T$.  It is easy to see from there that any added edge creates a copy of $T$, which gives us our upper bound.  

This idea of \emph{disjoint subcubes} is our first major idea for constructing $T$-saturated subgraphs.  Since removing any edge from trees $T$ creates two distinct connected components, they are particularly vulnerable to this kind of attack.  

An especially interesting case of such trees is the path $P_k$, for which a bound can quickly be found using Theorem $\ref{thm:supahimportant}$.  We, by altering the simple construction that solely takes edges in disjoint subcubes, were able to improve this bound, obtaining the following result.
\begin{theorem}
\label{thm:satpaths}
Let $k$ be an integer greater than 4 and $i = \lfloor \log_2(k-1)\rfloor$. Then,
$$
\sat(Q_n, P_k) \leq \begin{cases}  (i+1) \frac{k-1}{2^{i+1}} 2^{n-1} &\textrm{if $k$ is odd} \\ 
\left(\frac{i(k-2)}{2^{i+1}} + 1 \right)2^{n-1} & \mbox{if $k$ is even}.
\end{cases}
$$
\end{theorem}
The main idea in the construction is that, by deleting a set number of vertices of the same parity from $Q_k$, we can construct, with a few tweaks, a $P_{2^{k-1} + j}$-saturated graph for $1 \le j \le 2^{k-1}$.  

In a similar manner, we were able to tweak the construction of Theorem $\ref{thm:supahimportant}$ in the case of generalized stars.  Defining $P_j(Q_k)$ as the maximum length of $j$ vertex-disjoint paths all emanating from some vertex $v$ in $Q_k$, we found the following:
\begin{theorem}
\label{thm:genstarthm}
Let $m$ and $k$ be positive integers and let $m^{\prime} = \left \lfloor \log_{2}(m-1) \right \rfloor$ and $j =\left \lceil \log_{2}(m - 2^{m^{\prime}})\right \rceil$. Then, if $m \le P_{k-1}(Q_{k-1})$, $\sat(Q_n, GS_{k, m}) \le (k+1+m^{\prime} + \frac{j}{2^{m^{\prime}}} )\cdot 2^{n-2}$. 
\end{theorem}
The construction here is specific to the case at hand (generalized stars), so we refer the reader to its proof; the important idea, regardless, is that beginning with a construction using disjoint subcubes and then altering it to improve the bound is a powerful tool.

The second important tool that we explore and implement is the Hamming code, which, amplified from $Q_{2^{i}-1}$ (in which it is perfect) to $Q_{2^{i} - 1 + j}$, allows us to construct perfect dominating sets of $Q_{j}$.  This is extremely useful because, after filling these dominating sets, we can then construct $(r-1)$-regular subgraphs of the remaining vertices, where $r$ is the minimum degree of the tree in question.  This, with a few tweaks, allows us to construct saturated subgraphs (and thus find upper bounds) for many trees.  

Before stating our general results obtained from using constructions based on the Hamming code, we first present an example of how the Hamming code is used in saturation-type problems by deriving an upper bound on the caterpillar $S_{k \times r}$.  Before we state this result, however, we need to prove an important lemma.

\begin{lemma}
\label{lem:bipartitehall}
For every $k$-regular bipartite graph $H$, there exists some subgraph
$G$ of $H$ that is $r$-regular and bipartite for all nonnegative integers $r \le k$.  
\end{lemma}
\begin{proof}
Note first that any subgraph $G$ of a bipartite graph is necessarily
bipartite, so we only need to find an $r$-regular subgraph.  To do this, we
invoke Hall's Theorem \cite{hall}.  By a simple application of this theorem, we can
find a perfect matching within our $k$-regular subgraph.  Removing all
edges in this perfect matching, we are left with a $(k-1)$-regular
bipartite subgraph.  We can repeat this process $k-r$ times, and thereby
end up with an $r$-regular, bipartite graph, as desired.
\end{proof}
\begin{theorem}
\label{thm:doublestar}
For all positive integers $k, r$ where $k \ge r$, $\sat(Q_n, S_{k \times r}) \le r \cdot 2^{n-1}$.
\end{theorem}
\begin{proof}
Let $\kappa$ be the smallest integer greater than $k$ of the form $2^{j}
- 1$.  Now, consider a Hamming code $C$ on
$Q_{\kappa}$, which is perfect by the definition of $\kappa$.  Using this Hamming code, we will construct a
$S_{\kappa \times r}$-saturated graph $H$ on $Q_{\kappa}$.  Begin by adding to $H$ all incident edges to $C$.  This creates $\frac{2^{\kappa}}{\kappa + 1}$ vertices with
degree $\kappa$.  Now, note that the induced subgraph $H^{\prime}$ of $H$ with vertex set $V(Q_n) \ \backslash \ C$ is $1$-regular and bipartite, as it is a subgraph of $Q_n$.  From here, our preconditions satisfied, we use Lemma
$\ref{lem:bipartitehall}$ in reverse to add perfect matchings to $H^{\prime}$ until it is $(r-1)$-regular.  Adding this to $H$, we have a subgraph in which
$\frac{2^{\kappa}}{\kappa + 1}$ vertices have degree $\kappa$ and
$\frac{\kappa \cdot 2^{\kappa}}{\kappa + 1}$ vertices have degree $r-1$. 

 From these properties, it is easy to
see that $H$ is $S_{\kappa \times r}$-saturated, as any non-edge must be
incident to some vertex with degree $r-1$, which in turn is always adjacent to some vertex in $C$ (because $C$ is a Hamming code) with degree $\kappa > k$, thereby creating $S_{\kappa \times
  r}$ , and $H$ does not originally contain $S_{\kappa \times
  r}$, as there are no two adjacent vertices with degree $r$ or greater.  Note, importantly, that $H$ also does not contain $S_{k \times r}$, as $k \ge r$.

To scale $H$ up to $Q_n$, we simply need to consider the
subgraph of $Q_{\kappa} \ \Box \ Q_{n- \kappa}$ in which each vertex (a $Q_k$) of
$Q_{n-\kappa}$ contains $H$, and there are no edges between $Q_{\kappa}$'s.  In this subgraph of $Q_n$, all non-edges must be incident to some vertex with degree $r-1$, except those incident to two vertices in Hamming codes. However, since any edge between degree $\kappa$ vertices also creates $S_{\kappa \times r}$, our subgraph remains saturated, giving us an upper bound of 
\[2^{n-\kappa} \cdot \left(\kappa \cdot \frac{2^{\kappa}}{\kappa + 1} + (r-2) \cdot
\frac{\kappa \cdot 2^{\kappa-1}}{\kappa + 1}\right) \le r
\cdot 2^{n-1}. \qedhere\]
\end{proof}

The ideas present in this proof can then be extended further to obtain the following general results on caterpillars.
\begin{theorem}
\label{thm:generalcat}
Given a caterpillar $S_{k_1 \times k_2 \times \cdots \times k_m}$, let $\emin\{k_1, k_2, \ldots, k_m\} = (k_j, k_{j+1})$.  Given that $m = 2^a + b$ for some integer $2\le b \le 2^a$, then, if $b - 1 < j < 2^{a}$ and $\max\{k_j, k_{j+1}\} \ge \left \lfloor \log_{2} m \right \rfloor$, $\sat(Q_n, S_{k_1 \times k_2 \times \cdots \times k_m}) \le \max\{k_j, k_{j+1}\}\cdot 2^{n-1}$.
\end{theorem}
\begin{theorem}
\label{thm:othergeneralcat}
Consider $S_{k_1 \times \cdots \times k_m}$ for $m = 2^{a}+1$.  If $\emin\{k_1, \ldots, k_m\} = (k_i, k_{i+1})$ for $i = 1$ or $i = m-1$ and $\max\{k_i, k_{i+1}\} \ge a$, then $\sat(Q_n, S_{k_1 \times \cdots \times k_m}) \le \max\{k_i, k_{i+1}\} \cdot 2^{n-1}$.
\end{theorem}
Interestingly enough, we need not stop there.  Perfect dominating sets can also be used to classify what we call \emph{very generalized stars}.  In particular, very generalized stars are generalized stars in which the vertices along each of the legs are themselves central nodes of stars.  Given the very generalized star $VGS_{k, m}$ and legs $1, 2, \ldots, k$, we denote the degree of each of the vertices on the legs by $k_{ij}$, where $i$ is the leg number and $j$ is the position of the vertex on that leg.  From this definition, we were able to use perfect dominating sets to find the following result:
\begin{theorem}
\label{thm:vgsup}
Let $r = \min\{\max\{k_{i1}, k_{i2}\} :  1 \le i \le k\}$.  Then, given that there exists a pair of $k_{i1}, k_{j1} \ge r$ for $i \ne j$, that $m \le P_{k-1}(Q_{k-1})$, and that $r \ge k$, $\sat(Q_n, VGS_{k, m}) \le r \cdot 2^{n-1}$. 
\end{theorem}
We were also able to find a generalized version of this result concerning the class of trees $\mathcal{A}$ that satisfy the following property: given some central vertex and the subgraph $P$ of $A \in \mathcal{A}$ containing only vertices at most distance two from this central vertex and all edges incident to these vertices, all trees $a \in \mathcal{A}$ contain a leg (denoted by $\mathcal{L}$) beginning with a path of length two whose vertices $v_1, v_2$ satisfy $\emin(P) = \max\{\deg(v_1), \deg(v_2)\}$.  For this general class of trees, we get a similar result using essentially the same construction:
\begin{theorem}
\label{thm:asymmetricgeneral}
Let $\mathcal{A}'$ be the tree created by removing $\mathcal{L}$ from $\mathcal{A}$.  Then, given that there exists some two vertices adjacent to the central vertex with degree greater than $r$, that $\cd(\mathcal{A}'), \cd(\mathcal{A} - \mathcal{A}') \le k-1$, and that $r \ge \cd(\mathcal{A}') + 1$, we have that $\sat(Q_n, \mathcal{A}) \le r \cdot 2^{n-1}$.
\end{theorem}
As a whole, the plethora of general results giving upper bounds on the saturation number of trees, all based on the minimum degree of the forbidden graph in question, suggest that perfect binary codes and dominating sets may indeed allow us to find tight upper bounds for all trees.  It, along with the idea of constructing saturated subgraphs using disjoint subcubes, are our two most promising future directions.

Now, in the remaining sections, we present proofs of all of our results (some of which were not stated here).

%%%%%%%%%%%%%%%%%%%%%%%%%%%%%%%%%%%%%%%%%%%
%%%%%%%%%%%%%%%%%%%%%%%%%%%%%%%%%%%%%%%%%%%

\section{General Bounds and Methods}
\label{sec:genresults}
\vspace{-0.1 in}
In this section, we give proofs of our three general results that apply regardless of whether the forbidden graph is a tree.  The first of these is the previously described general lower bound on $\sat(Q_n, G)$.

\subsection{Proof of Theorem \ref{thm:lowerbound}}
\begin{proof}
Consider a subgraph $H$ of $Q_n$ that is $G$-saturated, and define $V_i$ as the number of vertices $h \in H$ with $\deg(h) = i$. For the sake of brevity, for $k \in \{0, \dotsc, n\}$, let $V_{\geq k}$ be the number of vertices of $H$ of degree at least $k$; in other words, $V_{\geq k} = V_k + \cdots + V_n$. Notice that the addition of any non-edge $uv$ of $Q_n$ must create a copy of $G$, and therefore at least one of $u$ and $v$ has a degree of at least $\delta-1$. 

Now, consider the number of pairs $(v, u)$ where $v$ is a vertex with degree less than $\delta - 1$ such that $uv \in E(Q_n) \ \backslash E(H)$. We know that the degree of $u$ must be at least $\delta - 1$. For each $v$ of degree $i < \delta-1$, there are exactly $(n-i)$ such $u$'s, so the number of pairs is 
\begin{equation}
\label{eq:vk}
\sum\limits_{i = 0}^{\delta-2}(n-i)V_i. 
\end{equation}
On the other hand, each vertex $u$ of degree at least $\delta-1$ is counted at most $n - (\delta -1)$
times in $(\ref{eq:vk})$. So
\begin{equation}
\label{summationinquestion2}
\sum\limits_{i = 0}^{\delta-2}(n-i)V_i \le
(n-\delta+1)V_{\geq \delta-1}.  
\end{equation} Since $\displaystyle\sum\limits_{i=0}^{n}iV_i = 2e(H)$, we have that
\begin{equation}
\label{eq:manipulation} 
\displaystyle\sum\limits_{i = 0}^{\delta-2}(n-i)V_i 
= n(2^n - V_{\geq \delta - 1}) - 2e(H) + \sum_{i=\delta-1}^n iV_i.
\end{equation}
Notice that the last term is at least $(\delta - 1)V_{\geq \delta-1}$, so, with $(\ref{summationinquestion2})$, we have that
\begin{equation}
\label{eq:inequality} 
n2^n \le 2e(H) + 2(n-\delta+1)V_{\geq \delta-1}.
\end{equation}
To bound $V_{\geq \delta-1}$, observe that
\begin{equation}
2e(H) \ge \displaystyle\sum\limits_{i=\delta-1}^{n} iV_i \ge (\delta-1)V_{\geq \delta-1}, \nonumber
\end{equation}
so $V_{\geq \delta-1} \le \frac{2e(H)}{\delta-1}$. Hence
\[n2^n \le \frac{4(n-\delta+1)}{\delta-1}e(H),\]
or simply,
\[e(H) \ge \frac{\delta-1}{2n-\delta+1}n2^{n-1} = (\delta-1 +
o(1))2^{n-2}, \]
as desired.
\end{proof}
\subsection{Other Preliminary Lemmata}
Our next two preliminary results are useful lemmata to be used later in the paper.  The first of these sets the foundation for some of our
inductive constructions in Section $\ref{sec:treesdisjoint}$.  In essence, it allows us to
classify the situations in which we can scale a saturated graph in $Q_{i}$
up to $Q_n$ while still maintaining saturation.  

Before stating this lemma, however, we first need to make an important definition.

%%%%%%%%%%%%%%%%%%%%%%%%%%%%%%%%%%%%%%%%
\begin{definition}
Given a graph $H$, a vertex $v \in V(H)$ is an \emph{endpoint} with respect to a forbidden graph $G$ if the addition of an incident edge to $v$ creates a new copy of $G$. 
\end{definition}
%%%%%%%%%%%%%%%%%%%%%%%%%%%%%%%%%%%%%%%%

\begin{lemma}
\label{lem:endpoints} 
Given a bipartite graph $G$, let $H$ be a $G$-saturated subgraph of $Q_k$ with $c \cdot e(Q_k)$ edges for some $c \leq 1$, and let $U \subseteq V(H)$ be the set of endpoints of $H$ with respect to $G$. Then, if there exists an automorphism of $Q_k$ which maps $V(H) \setminus U$ into $U$, $\sat(Q_n, G)\le c \cdot k 2^{n-1}$.
\end{lemma}
\begin{proof}
Let $H_0$ be such a saturated subgraph of $Q_k$. By assumption, there is a isomorphic copy $H_1$ of $H_0$ in $Q_k$ whose set of endpoints with respect to $G$ contain all of the non-endpoints of $H_0$. If we place $H_0$ and $H_1$ in disjoint $Q_k$'s in $Q_{k+1}$, then any edge between the two $Q_k$'s must be incident to either an endpoint of $H_0$ or an endpoint of $H_1$.  Therefore, the addition of that edge creates a copy of $G$. Furthermore, $H_0$ and $H_1$ are both themselves $G$-saturated, so we have constructed a $G$-saturated subgraph of $Q_{k+1}$. Similarly, consider $Q_n$ as $Q_{k} \ \Box \ Q_{n-k}$. Note that each $Q_k$ can be written as $(Q_k, x)$ where $x \in \{0, 1\}^{n-k}$. To construct our saturated graph in $Q_n$, place $H_0$ in $(Q_k, x)$ if $w(x)$ is even and $H_1$ if $w(x)$ is odd. This graph is again $G$-saturated by the same logic, so we have that $\sat(Q_n, G) \le
e(H_0) \cdot 2^{n-k}= c \cdot k2^{n-1}$, as desired.
\end{proof}

%%%%%%%%%%%%%%%%%%%%%%%%%%%%%%%
Finally, the following lemma, a generalized version of an observation of Johnson and Pinto \cite{JohnsonPinto}, is useful in later constructions. We omit its proof because it is relatively straightforward.
\begin{lemma}
\label{lem:greedy}
Suppose that $H^{\prime}$ is a $G$-free subgraph of $Q_n$ and that $S$ is the set of edges in $E(Q_n) \ \backslash  \ E(H^{\prime})$ that do not create a copy of $G$ when added to $H^{\prime}$.  Then, we can form a $G$-saturated subgraph $H$ of $Q_n$ by adding no more than $e(S)$ edges to $H^{\prime}$.
\end{lemma}
%%%%%%%%%%%%%%%%%%%%%%%%%%%%%%

%%%%%%%%%%%%%%%%%%%%%%%%%%%%%%%%%%%%%%%%%%%
%%%%%%%%%%%%%%%%%%%%%%%%%%%%%%%%%%%%%%%%%%%

\section{Bounds on $\sat(Q_n, T_k)$ using Disjoint Subcubes}
\label{sec:treesdisjoint}
In this section, we present proofs of our upper bounds on the saturation number of trees that use constructions based on disjoint subcubes.  We begin by examining bounds for general trees in Section $5.1$, and then derive tighter bounds for some special cases (paths and generalized stars) in Sections $5.2$ and $5.3$.  
\subsection{General Case}
Here, we prove Theorem $\ref{thm:supahimportant}$, our general upper bound on the saturation number of trees based on their cubical dimension.
\subsubsection{Proof of Theorem \ref{thm:supahimportant}}
\begin{proof}
Let $uv \in T$ be the edge which attains this minimum, and let $T_1$ and $T_2$ be the two connected components of $T\setminus uv$ where $T_1$ is rooted at $u$ and $T_2$ is rooted at $v$. Without loss of generality, let $\cd(T_1) = k$ and $\cd(T_2) = j$ for $j \le k$. Now, consider a subgraph of $Q_{k} \ \Box \ Q_{n-k}$. We claim that the subgraph consisting of only edges within each copy of $Q_k$ is saturated. Because $\cd(T)>k$, we see that this subgraph is $T$-free. Now, consider the addition of a non-edge $u'v'$ between two $Q_k$'s. Due to rotational symmetry, we can find an isomorphic copy of $T_1$ rooted at $u'$ in one $Q_k$ and an isomorphic copy of $T_2$ rooted at $v'$ in the other $Q_k$. This creates a copy of $T$, so our subgraph is $T$-saturated. All that remains is to enumerate the number of edges in our saturated subgraph, $2^{n-k} \cdot k \cdot 2^{k-1} = k \cdot 2^{n-1}$. This implies that $\sat(Q_n, T) \leq k \cdot 2^{n-1}$, as desired.
\end{proof}
\noindent Notice that the proof of Theorem $\ref{thm:supahimportant}$ also holds for any graph with a cut edge. That
is, if we define $B(G)$ as the set of edges in a graph such that if $e \in B(G)$ is removed, two disjoint connected components would be created, then the
same conclusion follows.   

\subsection{Paths}
We now prove our upper bound on the saturation number of paths in the hypercube. To begin, we present a simple upper bound using Theorem $\ref{thm:supahimportant}$.

\begin{proposition}
\label{prop:path}
$\sat(Q_n, P_k) \leq \lfloor \log_2 k \rfloor \cdot 2^{n-1}$.
\end{proposition}

\begin{proof}
Notice that $\cd(P_k)$ is $\lfloor \log_2 k \rfloor + 1$. If $k$ is odd, then $P_k$ without the $\frac{k+1}{2}$th edge is simply the graph consisting of two disjoint copies of $P_{\frac{k-1}{2}}$. Since $\cd(P_{\frac{k-1}{2}}) = \lfloor \log_2 k-1 \rfloor = \lfloor \log_2 k \rfloor$, by Theorem $\ref{thm:supahimportant}$ we have that $\sat(Q_n, P_k) \leq \lfloor \log_2 k \rfloor \cdot 2^{n-1}$. The case where $k$ is even is similar. 
\end{proof}

To improve this bound, we first need several lemmata on the maximum length of a path in $Q_k$ after the deletion of some set of vertices of the same parity.

\begin{lemma}
\label{lem:maxpath}
Let $k$ be an integer greater than 1 and $j$ be an integer such that $0 < j \le 2^{k-1}$. Then, for vertices $v_1, \dotsc, v_j \in V(Q_k)$ of the same parity, the maximum length of a path in $Q_k \setminus \{v_1, \dotsc, v_j\}$ is $2^{k} - 2j$. Moreover, let $a_{k-1}2^{k-2} + \dotsc + a_1 2^0$ be the binary representation of $j-1$, and let $C$ be the collection of even vertices that either have $0^{k-1-i}1$ as a prefix for some $i$ with $a_i = 1$ or are $0^k$. Then $|C| = j$ and there is a path of length $2^k - 2j$ in $Q_k \setminus C$. 
\end{lemma}
\begin{proof}
Let $v_1, \dotsc, v_j$ be vertices of $Q_k$ of the same parity. Without loss of generality, we can assume that this parity is even. Let $u_0 u_1 \dotsc u_\ell$ be a path of length $\ell$ in $Q_k \setminus \{v_1, \dotsc, v_j\}$. Clearly, the parity of vertices on this path must be alternating in the order that they appear in the path. But we have only $2^{k-1} - j$ even vertices, so the number of vertices on the path is at most $(2^{k-1}-j)+(2^{k-1}-j+1) = 2^k - 2j+1$ where the maximum is attained only if $\ell$ is even and $u_0, u_2, \dotsc, u_{\ell}$ are odd vertices. Hence $\ell \leq 2^k - 2j$. 

For the second part, note that the number of even vertices with the prefix $0^{k-1-i}1$ is $2^{i-1}$. So $|C| = 1+\sum_{i=1}^{k-1} a_i 2^{i-1} = j$. To construct a path of length $2^k - 2j$, let $i_1 > i_2 > \dotsc > i_s$ be indices such that $a_{i_t} = 0$. Starting from $v_0 = 0^{k-1}1$, we can jump to $(0^{k-1-i_s}1, Q_{i_s})$ by flipping the $k-i_s$th coordinate of $v_0$ and then taking a hamiltonian path inside the cube $Q_{i_s}$. Let $v_1$ be the endpoint of the current path. Again, from here, we can jump to $(0^{k-1-i_{s-1}}1, Q_{i_{s-1}})$ by flipping $k-i_{s-1}$th coordinate of $v_1$. We proceed in this manner to use all of the vertices of the$Q_{i_t}$. The number of vertices on our path is $1+\sum_t 2^{i_t} = 1+\sum_{i} (1-a_i)2^i = 2^k - 2j + 1$, so we have a path of length $2^k - 2j$. 
\end{proof}

\begin{lemma}
\label{lem:maxpath2}
Let $C \subset \{0, 1\}^{k}$ as in the previous lemma. If $j \leq 2^{k-2}$ ($a_{k-1} = 0$), then for any odd vertex $v$ in $Q_k$ there is a path of length $2^k-2j$ in $Q_k \setminus C$ starting from $v$.
\end{lemma}

\begin{proof}
We first prove this for the case where $v$ is in $(0^{k-i-1}1, Q_i)$ for $a_i = 0$. In this case, we begin with a hamiltonian path starting at $v$ and ending at $0^{k-i-1}10^{i-1}1$, which exists since $v$ is odd but $0^{k-i-1}10^{i-1}1$ is even. From there, we jump to $0^{k-1}1$, and, as in the proof of the previous lemma, construct a path from $0^{k-1}1$ to traverse all vertices in $(0^{k-i'-1}1, Q_{i'})$ for all $i'\neq i$ with $a_{i'} = 0$. This path, as before, has length $2^k - 2j$. 

For the other case, we use the fact that $a_{k-1} = 0$, so no vertex whose first coordinate is 1 was removed in our construction. If $v=0^{k-i-1}1w$ for some $i$ with $a_i = 1$ and $w \in \{0, 1\}^i$, then we jump to $10^{k-i-2}1w$ in $(1, Q_{k-1})$.  This must necessarily be even since $v$ is odd; from there, we take a path of length $2^{k-1} - 2$ to $10^{k-2}1$. Then, as before, we jump to $0^{k-1}1$ and use the same argument to obtain a path of length $2^k - 2j$. 
\end{proof}

Using these lemmata, we can now prove Theorem $\ref{thm:satpaths}$.
\begin{proof}[Proof of Theorem \ref{thm:satpaths}]
Let $k = 2^i+r$ for $0 < r \leq 2^i$. We first prove the case in which $r$ is odd. Let $j = 2^{i-1} - \frac{r-1}{2}$ so $0 < j \leq 2^{i-1}$, and let $C \subset Q_{i+1}$ be the set as in Lemma \ref{lem:maxpath} with $|C|=j$. We claim that $H = Q_{i+1} \setminus C$ is $P_k$-saturated. It is clear that the maximum length of a path in $H$ is $2^{i+1} - 2j = 2^i + r - 1 = k-1$, so $H$ is $P_k$-free. Furthermore, by Lemma \ref{lem:maxpath2}, any odd vertex in $H$ is an endpoint of a copy of $P_{k-1}$. Since any non-edge is incident to both an odd vertex and an even vertex in $C$, the addition of any non-edge creates a copy of $P_k$. Hence $H$ is $P_k$-saturated. 

All that remains, then, is to compute the number of edges in $H$, 
$$
(i+1)2^i - (i+1)j = (i+1)\left(2^{i-1}+\frac{r-1}{2}\right) = \frac{i+1}{2} \cdot \frac{k-1}{2^{i}} \cdot 2^{i}.
$$
From here, by Lemma \ref{lem:endpoints}, we get that
$$
\sat(Q_n, P_k) \leq \frac{i+1}{2} \cdot \frac{k-1}{2^{i}} \cdot 2^{n-1} = (\lfloor \log_2 k \rfloor + 1)\cdot \frac{k-1}{2^{\lfloor \log_2 k \rfloor + 1}} \cdot 2^{n-1},
$$
as desired.

In the case where $r$ is even, let $j = 2^{i-1} - \frac{r-2}{2}$, and similarly let $C \subset \{0, 1\}^{i+1}$ be the set of even vertices of $Q_{i+1}$ as in Lemma \ref{lem:maxpath} with $|C| = j$. Now, let $C'$ be the set of even vertices in $Q_i$ obtained by flipping the first two coordinates of all of the vertices in $C$. Since the first coordinate of every vertex in $C$ is 0, $C \cap C' = \emptyset$. From here, let $H$ be the subgraph of $Q_{i+2}$ such that $E(H)$ consists of all edges in $(0, Q_{i+1} \setminus C)$, all edges in $(1, Q_{i+1} \setminus C')$ and all edges connecting $(0, Q_{i+1})$ and $(1, Q_{i+1})$ which are incident to either $(0, C)$ or $(1, C')$. We claim that $H$ is $P_k$-saturated.

Note that $H$ has two connected components $(0, C) \cup (1, \{0, 1\}^{i+1} \setminus C')$ and $(0, \{0, 1\}^{i+1} \setminus C) \cup (1, C')$. If $H$ contains a path of length $k$, then it must lie in one of the components; let this component be $(0, \{0, 1\}^{i+1} \setminus C) \cup (1, C')$. We know that the maximum length of a path in $Q_{i+1}\setminus C$ is $2^{i+1}-2j = 2^i + r - 2 = k-2$. Since any vertex in $(1, C')$ has only one incident edge in $H$, if there is a path of length $k$ in $H$ then two endpoints of the paths must be in $(1, C')$. Those vertices have same parity, so the length of the path must be odd, contradicting that $k$ is even. Therefore, $H$ is $P_k$-free.

To show saturation, let $uv$ be a non-edge of $H$. Then, $uv$ is either incident to $(0, C) \cup (1, C')$ or $u$ is a vertex in $(0, \{0, 1\}^{i+1}\setminus (C \cup C'))$ and $v$ is a vertex in $(1, \{0, 1\}\setminus (C \cup C'))$. For the first case, without loss of generality assume $v \in (1, C')$. Then $u$ is an even vertex in $(1, \{0, 1\}^{i+1})$, so there is a path of length $k-2$ starting at $u$ in $(1, Q_{i+1} \setminus C')$. On the other hand, there is an edge in $H$ joining $v$ and $(0, \{0, 1\}^{i+1})$ so, together with $uv$, we can construct a path of length $k$. For the second case, suppose that $u$ is an odd vertex. Then there is a path of length $k-2$ in $(0, Q_{i+1}\setminus C)$ starting at $u$. Similarly there is a path of length $k-2$ in the other component starting at $v$, so the addition of $uv$ creates a path of length $2k-3 \geq k$. Finally, if $u$ is an even vertex, we can construct a path of length $k-3$ from $u$ in $(0, Q_{i+1} \setminus C)$ using a similar construction as in the proof of Lemma \ref{lem:maxpath2}. Therefore, by symmetry the addition of $uv$ necessarily creates a path of length $2k-5 \geq k$, so $H$ is $P_k$-saturated.

The number of edges in $H$ is
$$
2((i+1)2^i - (i+1)j) + 2j = \left(\frac{i+2}{2} + \frac{i(r-2)}{2^{i+1}}\right)2^{i+1}.
$$
Thus by Lemma \ref{lem:endpoints}, we get that
$$
\sat(Q_n, P_k) \leq \left(\frac{i+2}{2} + \frac{i(r-2)}{2^{i+1}}\right)2^{n-1} = \lfloor \log_2 k \rfloor \left(\frac{k-2}{2^{\lfloor \log_2 k \rfloor+1}} + \frac{1}{\lfloor \log_2 k\rfloor}\right) 2^{n-1}. \qedhere
$$
\end{proof}

\subsection{Generalized Stars}
Next, we study generalized stars.  The first step in this process is determining the value of $P_j(Q_k)$, which leads us to the following lemma.
\begin{lemma}
\label{lem:edgedisjointpaths}
Given the hypercube $Q_k$ and $j \le k$, $P_j(Q_k) \ge k-1$.
\end{lemma}
\begin{proof}
Let us denote the directions in the $k$-dimensional hypercube by $1$, $2$, $\ldots$, $k$.  We characterize each path in $Q_m$ of length $m$ by a $m$-tuple $(a_1, a_2, \ldots, a_m)$, where $a_i \in \{1, 2, \ldots, k\}$, representing the order of directions travelled.

Consider $k$ paths starting at $v$, $\mathcal{P}_1 = (1, 2, \ldots, k-1)$, $\mathcal{P}_2 = (2, 3, \ldots, k)$, $\ldots$, $\mathcal{P}_k = (k, 1, \ldots, k-2)$.  Refer to the $i$-tuple corresponding to the first $i$ directions in $\mathcal{P}_j$ by $\mathcal{P}_{ji}$.  Now, note that, in these paths, regardless of order, $\mathcal{P}_{ai} \ne P_{bi}$ for all $i \in \{1, 2, \ldots, k\}$, $a \ne b$.  This implies that the paths must be vertex-disjoint and, since they are all of length $k-1$, we are done.  
\end{proof}

However, we conjecture that in fact $P_j(Q_k)$ can be much larger.  Let $\mathcal{E}(n) = \displaystyle\sum\limits_{k=2}^{\left \lfloor \frac{n}{2} \right \rfloor} \binom{n}{k}$ and $\mathcal{O}(n) = 2^n - \mathcal{E}(n) - 1$.  Then we conjecture the following:
\begin{conjecture}
Let $a = \left \lfloor \frac{\mathcal{E}(n)}{j} \right \rfloor$ and $b = \left \lfloor \frac{\mathcal{O}(n)}{j} \right \rfloor$.  Then, if $\min\{a, b\} = b$, $P_j(Q_k) = 2b$.  Otherwise, $P_j(Q_k) = 2a+1$.  
\end{conjecture}

We have been able to show this conjecture for $k = 2 - 8$ and hope to prove it completely in the future.  Regardless, we were able to use this definition to prove Theorem \ref{thm:genstarthm}, our upper bound on the saturation number of generalized stars.
\subsubsection{Proof of Theorem \ref{thm:genstarthm}}
\begin{proof}
We begin by constructing a $GS_{k, m}$-saturated graph in $Q_k$, which we denote by $H$.  Within $Q_k$, consider two $(k-1)$-dimensional subcubes, $A = (0, \{0, 1\}^{k-1})$ and $B =  (1, \{0, 1\}^{k-1})$.  We start by adding all of the edges in $A$ to $H$.  

Next, we split $B$ into $2^{k-1-m^{\prime}}$ disjoint $m^{\prime}$-dimensional hypercubes $b_1, b_2, \ldots, b_{k-1-m^{\prime}}$, and add all of the edges within the $b_i$ to $H$.  Each of the $b_i$ can be represented by $(e_1, e_2, \ldots, e_{t-1}, \{0, 1\}^{m^{\prime}}, \\ e_{t}, \ldots, e_{k-m^{\prime}})$, for some set of $e_j \in \{0, 1\}$.   In each of these, consider the vertex $(e_1, e_2, \ldots, e_{t-1}, 0, \\0, \ldots, 0, e_{t}, \ldots, e_{k-m^{\prime}})$.  This set of vertices comprises a $(k-1-m^{\prime})$-dimensional hypercube $C$.  Within $C$, consider a set of disjoint $2^{k-1-m^{\prime}-j}$ $j$-dimensional hypercubes $c_1, c_2, \ldots, c_{k-1-m^{\prime}-j}$, and add to $H$ a path of length $2^{j}-1$ in each of the $c_i$.  Notice that this path covers all of the vertices in each of the $c_i$.

From here, let $f = m - 2^{m^{\prime}}$, and, in each of the $c_i$, let $c_{i1}$ denote the vertex for which the longest path starting at this vertex is $f-1$.  Note that there must exist such a vertex, as $2^{j-1} < f \le 2^{j}$.  From each of these $c_{i1}$, add to $H$ the edge between it and $A$.  Then, within each of the $c_i$, use Lemma $\ref{lem:greedy}$ to greedily add edges until they are $c_{i1}$ $P_f$-saturated.  In other words, the addition of any new edge within these $c_i$ should create a copy of $P_f$ starting at $c_{i1}$.  We exclude the case $f = 2^{j}$, and therefore this is possible since each $c_i$ evidently can contain a path of length $2^{j}-1$.  For $f = 2^j$, we simply add every edge in $c_{i1}$, and therefore there are no non-edges left within the $c_i$ to worry about when considering saturation.  Note that we can bound above the number of edges added in each $c_i$ by $j2^{j-1}$.

Now, we claim that $H$ is almost saturated.  In particular, we claim that $H$ is free of $GS_{k, m}$, and that the addition of any edge in $B$ creates a copy of $GS_{k, m}$.  We first show the former.  Note that the central vertex of a $GS_{k, m}$ in $H$, if there were one, would have to be in $A$, since no vertex in $B$ has degree $k$.  Since our construction is symmetrical around connected vertices, let us choose an arbitrary vertex in $A$ that is connected to $B$.  From there, since $m \le P_{k-1}(Q_{k-1})$, we can fit $k-1$ legs of our $GS_{k, m}$ in $A$.  The final leg must traverse as its first edge the edge connecting it to $B$.  It is now in both a $c_i$ and a $b_i$.  Note that by construction no two $c_i$ are connected, so the only edges this path can use are the edges of the $c_i$ and the corresponding $2^{j}$ $b_i$.  However, also note that the $b_i$ are only connected by at most one edge, and therefore, once our path uses edges of a $b_i$, it cannot return to traversing the edges of a $c_i$.  It follows that the second and following edges, to achieve the maximal possible path, must be on a $c_i$.  The maximum possible number of edges in such a path, by construction, is $f - 1 = m - 2^{m^{\prime}} -1$.  From there, the only remaining edges that can be added lie on the $b_i$ on which the end of the path is currently located.  The maximum possible length of such a path is $2^{m^{\prime}}-1$, since it is in a $m^{\prime}$-dimensional hypercube.  Therefore, the maximum length of the final leg in $H$ is $1 + m-2^{m^{\prime}}-1 + 2^{m^{\prime}}-1 = m-1$, so $H$ is $GS_{k, m}$-free.  For the sake of rigor, note that there are no connecting edges back to $A$ along our possible paths once we enter $B$, and therefore our maximal path length is preserved.  

Next, we must show that the addition of any edge within $B$ creates a copy of $GS_{k, m}$.  First note that there are three types of non-edges in $B$: non-edges within a $c_i$, non-edges between $c_i$'s, and non-edges between $b_i$'s but not between or within $c_i$'s.  

\textbf{Case 1:} non-edges within a $c_i$

If the edge is within a $c_i$, we know by construction this creates a $P_f$.  By the same maximal construction as described when we showed that originally $H$ was $GS_{k, m}$-free, we now have a maximal leg length of $1 + f + 2^{m^{\prime}}-1 = 1 + m-2^{m^{\prime}} + 2^{m^{\prime}}-1 = m$, and therefore we have a copy of $GS_{k, m}$.  

\textbf{Case 2:} non-edges between $c_i$'s

In this case, let us consider again the final leg of our generalized star. This leg consists of an edge between $A$ and $B$, a path on the current $c_i$ to the vertex that is incident to the added non-edge, the new edge, the remaining path on the new $c_i$, and a path of length $2^{m^{\prime}}-1$.  Now, note that our path within individual $c_i$'s is broken up by the new edge between them, but the path itself at minimum has the same length as the longest path starting at a  $c_{i1}$, which we know to be $f-1$.  This gives us an overall leg length of $2 + m-2^{m^{\prime}}-1 + 2^{m^{\prime}}-1 = m$, thereby creating a copy of $GS_{k, m}$.  

\textbf{Case 3:} non-edges between $b_i$'s

Here, again, all that concerns us is the final leg, which we begin by traversing the $1$st direction into the $c_i$ in $B$ that contains the $b_i$ containing a vertex incident to our newly added edge.   From there, we traverse a path within the $c_i$ to this $b_i$, traverse a path in this $b_i$ of length at least $2^{m^{\prime}}-2$ ending at the incident vertex, traverse the new edge, and then traverse a path of length $2^{m^{\prime}}-2$ to the vertex in this new $b_i$ that is also contained within a $c_i$.  From there, we traverse the path from this vertex on the $c_i$ of length at least $\left \lceil \frac{f}{2} \right \rceil$.  At minimum, this path is of length $2 + 2^{m^{\prime}} - 2 + 2^{m^{\prime}} - 2 + \left \lceil \frac{f}{2} \right \rceil = 2^{m^{\prime} +1}-2 + \left \lceil \frac{f}{2} \right \rceil$.  The only case in which $2^{m^{\prime} + 1} - 2$ may not be sufficient is when $m = 2^{m^{\prime} + 1}$, but in this case $\left \lceil \frac{f}{2} \right \rceil = 2^{m^{\prime}} \ge 2$. Therefore, our newly created leg is sufficiently long and we have a copy of $GS_{k, m}$, so $H$ is almost saturated.

From there, the only remaining set of edges to deal with is the set of edges between $A$ and $B$.  In this case, we simply use Lemma \ref{lem:greedy} to greedily add edges until $H$ is saturated.  This adds a maximum of $2^{k-1}$ edges, and gives us that $H$ is saturated in $Q_k$.

To scale this up to $Q_n$, simply place copies of $H$ at the vertices of $Q_{k} \ \Box \ Q_{n-k}$.  Added edges between these copies either connect copies of $A$ or copies of $B$.  Edges connecting copies of $A$ evidently create $GS_{k, m}$, since we simply take $k-1$ legs in one copy of $A$ and the final one in the other, since it is a full $Q_{k-1}$.  For edges that connect copies of $B$, we can easily see that it is equivalent to the argument that edges within $B$ create $GS_{k, m}$, since we are simply connecting vertices in disjoint $c_i$ or $b_i$.  Therefore, all that remains is to compute the maximum number of edges in $H$, which gives us an upper bound of 
\[2^{n-k} \cdot \left((k-1)2^{k-2} + 2^{k-1} + 2^{k-1-m^{\prime}} \cdot m^{\prime}2^{m^{\prime}-1} + 2^{k-1-m^{\prime}-j} \cdot j2^{j-1}\right) = (k+1+m^{\prime} + \frac{j}{2^{m^{\prime}}} )\cdot 2^{n-2}.   \]
\end{proof}
Notice that this proof still applies even if the generalized star is not \emph{balanced}: that is, if the legs of the generalized star are of different lengths.  We can see this because all of the unbalanced cases reduce to the class of unbalanced generalized stars in which all but one leg are of the maximum length in the original generalized star, and the final leg is of the minimum length.  We then are able to place all but one leg in $Q_{k-1}$ and use a similar construction as in Theorem $\ref{thm:genstarthm}$  to create a saturated graph for the remaining leg.  
\vspace{-0.12 in}

%%%%%%%%%%%%%%%%%%%%%%%%%%%%%%%%%%%%%%%%%%%
%%%%%%%%%%%%%%%%%%%%%%%%%%%%%%%%%%%%%%%%%%%

\section{Bounds on $\sat(Q_n, T_k)$ using the Hamming code}
\label{sec:treeshamming}
In this section, we prove our upper bounds on the saturation number of trees whose constructions are based on the Hamming code.  We first prove the special case of stars.  Then, in Section $6.2$, we extend this argument  to derive our general upper bound on caterpillars with large minimum degree, and then show how this can be extended to other classes of trees in Section $6.3$.
\subsection{Stars}

First, we prove an upper bound for the specific case of stars. Note that it is easy to get the bound $\sat(Q_n, S_k) \leq (k-1)2^{n-1}$ by Theorem $\ref{thm:supahimportant}$, as the removal of any edge from a star
with $k$ edges lowers the cubical dimension.   We improve this bound.

\begin{theorem}
\label{thm:upstars}
Given the star $S_k$, $\sat(Q_n, S_k) \le (k - 2 + o_{k}(1))2^{n-1}$.  
\end{theorem}

\begin{proof}
Let $k = 2^t - 1 + m$ with $0 \leq m < 2^t$. Theorem \ref{thm:hamming_cube} implies that there is a perfect dominating set $C$ of $Q_k$ such that $C$ induces a subgraph of $Q_k$ which is a disjoint union of $\frac{2^{k-m}}{k-m+1}$ copies of $Q_m$. Let $H$ be the subgraph of $Q_k$ consisting of edges within $C$ or within $V(Q_k)\setminus C$. We claim that $H$ is $S_k$-saturated.

We first prove that $H$ is $S_k$-free, or equivalently, $\deg_H(u) \leq k-1$ for all $u \in V(Q_k)$. If $u \in C$, then by construction $\deg_H(u) = m \leq k-1$. On the other hand, if $u \not\in C$ then $\deg_H(u) = k-1$ since $u$ is adjacent to exactly one vertex in $C$. These two facts imply that $H$ is $S_k$-saturated, since any non-edge $uv$ of $H$ must necessarily lie between $C$ and $V(Q_k) \setminus C$.

Note that the number of edges of $H$ is $e(Q_k) - |V(Q_k) \setminus C|$ since there is a one-to-one correspondence between non-edges of $H$ and $V(Q_k) \setminus C$. Thus, we have that
$$
e(H) = k2^{k-1} - \left(2^{k} - \frac{2^{k}}{k-m+1}\right) = \left(k - 2 + \frac{2}{k-m+1}\right)2^{k-1}.
$$
From here, we use Lemma \ref{lem:endpoints} to get that
$$
\sat(Q_n, S_k) \leq \left(k - 2 + \frac{2}{k-m+1} \right) 2^{n-1}.
$$
Since $k-m+1 = 2^t \geq \frac{k+2}{2}$, we have from this an upper bound of $(k-2+o_k(1))2^{n-1}$. 
\end{proof}

\subsection{Caterpillars}
Next, we study $\sat(Q_n, S_{k_1 \times k_2 \times \cdots \times k_m})$.   We begin by deriving improved upper
bounds for all caterpillars with three and four vertices on the central path, and then prove our general upper bound for caterpillars with sufficiently large minimum degree.

For convenience, in these proofs, we denote the construction in which we create our saturated graph in a hypercube of dimension $\kappa$ such that we include all edge in and incident to a dominating set of $Q_i$'s as a \emph{$\kappa$-construction}.  Such constructions allows us to cover all but some set number of vertices in our caterpillar, since $\kappa$ can be chosen arbitrarily to be larger than $k_1, k_2, \ldots, k_m$.  

With that aside, we begin by finding bounds on $S_{k_1 \times k_2 \times k_3}$ and $S_{k_1 \times k_2 \times k_3 \times k_4}$ for completeness, as they are not completely covered by the general case.
\begin{theorem}
\label{thm:triplestar}
For all $k_1, k_2, k_3 \ge 2$, $\sat(Q_n, S_{k_1 \times k_2 \times k_3}) \le \min\{k_1, k_2, k_3\} \cdot 2^{n-1}$. 
\end{theorem}
\begin{proof}
The proof of this is very similar in construction to the cases of Theorem $\ref{thm:quadruplestar}$, so we leave it to the reader.
\end{proof}

\begin{theorem}
\label{thm:quadruplestar}
For all $k_1, k_2, k_3, k_4 \ge 3$, $\sat(Q_n, S_{k_1 \times k_2 \times k_3 \times k_4}) \le \min\{k_1, k_2, k_3, k_4\}\cdot 2^{n-1}$.
\end{theorem}
\begin{proof}
First, consider the cases where $k_2= \min\{k_1, k_2, k_3, k_4\}$ or $k_3 = \min\{k_1, k_2, k_3, k_4\}$ (these are equivalent by symmetry), and either way let $r = \min\{k_1, k_2, k_3, k_4\}$.  Within these cases, we will use a $\kappa$-construction with $i = 0$, and will find a $S_{\kappa-1 \times \kappa-1 \times r \times \kappa-1}$-saturated graph, which in turn will imply our result.  Denote this saturated graph by $H$.  Begin by considering two Hamming codes on $Q_n$, $C$ and $D = C + v_1$, or a translated copy of $C$ in the $1$st direction.  Now, add to $H$ all edges incident to $C$ and $D$, thereby giving all of the vertices in $C \cup D$ degree $\kappa$.  From there, use Lemma $\ref{lem:bipartitehall}$ to find and add $r-3$ perfect matchings to $V(Q_n) \backslash (C \cup D)$, thereby creating an $(r-1)$-regular graph among the vertices not in $C \cup D$.  This finishes the construction of $H$.  

Now, we show that $H$ is saturated. It is evidently free of $S_{\kappa \times \kappa \times r \times \kappa}$, as there are no four vertices of degree $\kappa$ adjacent to one another in our construction. However, also note that it is free of $S_{k_1 \times k_2 \times k_3 \times k_4}$, since there are no four adjacent vertices of degree $r$ or greater.  Now, consider some non-edge in $Q_{\kappa}$.  If added, it must be incident to some vertex $v$, which has degree $r-1$.  This edge added, consider two neighbors of $v$, $c \in C$ and $d \in D$.  Note that $c$ also has a neighbor $d_0 \in D$ such that $d_0 \ne d$, as otherwise $c$ and $d_0$ could not be adjacent.  Now, consider the central path $d_0-c-v-d$, and all associated emanating edges.  This almost creates $S_{\kappa \times \kappa \times r \times \kappa}$, except we notice that $c$ and $d$ actually share a vertex as an endpoint of an emanating edge.  However, this is not a problem, since simply not considering this edge gives a copy of $S_{\kappa-1 \times \kappa - 1 \times r \times \kappa-1}$, which, since $\kappa$ can be chosen to be greater than $\max\{k_1, k_2, k_3, k_4\}$, implies saturation.

Finally, we need to scale this graph up to $Q_n$ from $Q_{\kappa}$, in which case we simply place our saturated graph $H$ at the vertices of $Q_{\kappa} \ \Box \ Q_{n - \kappa}$.  Note that we do not need to worry about respective vertices of our Hamming codes being adjacent, because adding edges between them creates an $S_{\kappa \times \kappa \times r \times \kappa}$ with central vertices $c_0 - c_1 - d_0 - d_1$. 

From here, all that remains is to enumerate the number of edges, which gives an upper bound of
\[(2\kappa - 1) \cdot \frac{2^{\kappa}}{\kappa +1} + (r-3) \cdot \frac{\kappa 2^{\kappa-1} - 2^{\kappa}}{\kappa + 1} \le r \cdot 2^{n-1}.\]

Now, consider the second case, in which $k_1= \min\{k_1, k_2, k_3, k_4\}$ or $k_4 = \min\{k_1, k_2, k_3, k_4\}$, and let $r = \min\{k_1, k_2, k_3, k_4\}$. We again use a $\kappa$-construction with $i = 0$, and will find a $S_{\kappa - 1 \times \kappa - 1 \times \kappa - 1 \times r}$-saturated graph within $Q_{\kappa}$.

Let us begin by denoting our saturated graph by $H$.  Now, consider three Hamming codes in $Q_{\kappa}$, say $C$, $D = C + v_1$, and $E = C + v_2$. From here, add all edges incident to the vertices in each of the Hamming codes.  Then, among the remaining vertices not in $C \cup D \cup E$, which at the moment form a $3$-regular bipartite graph, find $r-4$ perfect matchings using Lemma $\ref{lem:bipartitehall}$ and add these edges, thereby creating a subgraph $H$ where all vertices in $C \cup D \cup E$ have degree $\kappa$, and all other vertices have degree $r-1$.  

From here, all that remains is to show that $H$ is saturated in $Q_{\kappa}$, and then demonstrating that we can scale this construction up to $Q_n$.  To do the former, we first note that it is clearly $H$-free, since there are no paths of length three among vertices with degree $r$ or greater, since by construction our Hamming codes form paths of length no more than two. Furthermore, if we add any non-edge to $H$, it must necessarily be incident to some vertex $v$ with degree $r-1$.  Considering its neighbor $d \in D$, $d$'s neighbor $c \in C$, and $c$'s neighbor $e \in E$, it can be seen that $v-d-c-e$ creates a copy of $S_{\kappa \times \kappa \times \kappa \times r}$.  The only problem arises when the stars around $d$ and $e$ share some vertex, but since this can only occur at one vertex besides $v$, we can simply not consider these edges and still have $S_{\kappa-1 \times \kappa - 1 \times \kappa - 1 \times r}$ and therefore, by choice of $\kappa$, $S_{k_1 \times k_2 \times k_3 \times k_4}$.  Thus, $H$ is saturated.  To show that this construction can scale, simply place a copy of $H$ at every vertex of $Q_{\kappa} \ \Box \ Q_{n -\kappa}$.  Since connecting any two vertices in respective Hamming codes creates a path of length six with all vertices with degree at least $\kappa$ (sufficiently long for our caterpillar), this subgraph of $Q_n$ is $S_{k_1 \times k_2 \times k_3 \times k_4}$-saturated.

All that remains is to enumerate the number of edges in our graph, which gives us an upper bound of
\[(3\kappa - 2) \cdot \frac{2^{\kappa}}{\kappa +1} + (r-4) \cdot \frac{\kappa 2^{\kappa-1} - 2^{\kappa+1}}{\kappa + 1} \le r \cdot 2^{n-1}.\]
\end{proof}

Now, for our two results on general caterpillars, we present the proof of the former, Theorem \ref{thm:generalcat}; the proof of the latter is similar, so we omit it.
\subsubsection{Proof of Theorem \ref{thm:generalcat}}
\begin{proof}
Let $\max\{k_1, k_2, \ldots, k_m\} = k$, and $\max\{k_j, k_{j+1}\}=r$. We again use a $\kappa$-construction with $i = a$, choosing $\kappa > k+a-1$ for reasons that will become clear later.  We will show that we can find a $S_{\kappa-a+1\times \kappa-a+1 \times \cdots \times \kappa-a+1 \times r \times r \times \kappa-a+1 \times \cdots \times \kappa-a+1}$-saturated graph, which implies our result. 

Let us begin by denoting our saturated subgraph by $H$.  Now, let $j_0 = \max\{j, m - j\}$.  We know that $j_0 < 2^{a}-1$, so $P_{j_0}$ can be embedded in $Q_a$. Furthermore, consider our dominating set $S$ consisting of disjoint $Q_a$'s.  Within our saturated graph $H$, we first add all edges within and incident to $S$.  Then, we use Lemma $\ref{lem:bipartitehall}$ on the remaining vertices to consistently add perfect matchings that give all other vertices degree $r-1$.  However, we do this in such a way that every pair of adjacent vertices in $Q_{\kappa}$ that are also adjacent to the same $Q_a$ in $S$ is connected by an edge.  This requires that $r \ge a$, a condition already satisfied since $a = \left \lfloor \log_{2} m \right \rfloor$.  From here, we now have that every vertex in $S$ has degree $\kappa$, and the remaining vertices in $Q_{\kappa}$ have degree $r-1$. 

Given our construction of $H$, we claim that this subgraph $H$ is saturated with our given caterpillar.  To show this, first note that the longest path in any given $Q_a$ is $2^{a} - 1$.  From there, since no two adjacent vertices in $S_{k_1 \times k_2 \times \cdots \times k_m}$ both have degree $r-1$ or less, the maximum overall central path length in $H$ is $2^{a}+1$.  However, $m > 2^{a} + 1$, so $H$ is $S_{k_1 \times k_2 \times \cdots \times k_m}$-free.  

Now, consider the addition of a non-edge in $Q_{\kappa}$ to $H$.  Since we have already connected all adjacent vertices adjacent to the same disjoint $Q_a$, this edge must connect vertices, say $v$ and $w$, adjacent to different $Q_a$.  These vertices must also necessarily be of degree $r-1$ in $H$.  From here, let $a_{v}$ and $a_{w}$ be the adjacent vertices to $v$ and $w$ in $S$.  Now, consider a path of length $2^{a} - 3$ on the first $Q_a$ ending at $a_{v}$, say $a_1 - a_2 - a_3 - \cdots - a_{v}$, and a path of length $b +1$ in the second $Q_a$ starting at $a_{w}$.  Since the longest path in a $Q_a$ is $2^{a} - 1$, this is evidently possible to construct.  

We claim that, using these vertices and the edges surrounding them, it is possible to create $S_{\kappa-a+1\times \kappa-a+1 \times \cdots \times \kappa-a+1 \times r \times r \times \kappa-a+1 \times \cdots \times \kappa-a+1}$.  To do this, note, given any vertex adjacent to a vertex in our path, there can be at most $a-1$ other vertices in our path also adjacent to it.  Since caterpillars contain no cycles, we see now why $\kappa$ necessarily must be greater than $k+a-1$, as it allows us to keep all such edges vertex-disjoint if we limit the number of edges emanating from each vertex to $\kappa-a+1$, thereby creating $S_{\kappa-a+1\times \kappa-a+1 \times \cdots \times \kappa-a+1 \times r \times r \times \kappa-a+1 \times \cdots \times \kappa-a+1}$. Since $\kappa - a + 1> k $, this applies equally well to our original caterpillar.  Note further that it is clear now why this implies that the graph is $S_{k_1 \times k_2 \times \cdots \times k_m}$-saturated, since the same arguments about saturation apply to this caterpillar, given the minimum degree $r$ is in the right place in the sequence.

The final step is to scale $H$ up to $Q_n$. But, in this case, simply placing $H$ at all vertices of $Q_{\kappa} \ \Box \ Q_{n -\kappa}$ works, since any non-edge either connects two vertices with degree $r-1$ adjacent to different $Q_a$, in which case we are already done, or connects two $Q_a$.  In the latter case, we can construct a path of length $2^{a}- 1 + 1 + 2^{a} - 1 = 2^{a+1} - 1 \ge m$ with all vertices having degree at least $\kappa - a + 1 > k$, so a copy of our original caterpillar is constructed.  From here, all that remains is to enumerate the number of edges in our saturated subgraph of $Q_n$, which gives us, as desired, an upper bound of 
\[2^{n - \kappa} \cdot \left(a \cdot \frac{2^{\kappa - 1}}{\kappa - a + 1} + (\kappa - a) \cdot \frac{2^{\kappa}}{\kappa - a + 1}+(r-2)\cdot \frac{(\kappa - a + 1)\cdot 2^{\kappa-1} - 2^{\kappa-1}}{\kappa - a + 1} \right) \le r \cdot 2^{n-1}. \qedhere \]
\end{proof}

Note that the proof of Theorem $\ref{thm:generalcat}$ functions even if $\emin\{k_1, k_2, \ldots, k_m\} \ne (k_j, k_{j+1})$ as long as there do not exist two pairs $(k_c, k_{c+1})$ and $(k_{d}, k_{d+1})$, where $c\le b-1$ and $d \ge 2^a$, that satisfy $\max\{k_c, k_{c+1}\}\le\max\{k_i, k_{i+1}\}$ and $\max\{k_{d}, k_{d+1}\}\le\max\{k_i, k_{i+1}\}$. From this, it is easy to see that Theorem $\ref{thm:generalcat}$ gives a lower bound for many caterpillars with length $m \ne 2^{i}+1$, provided that the minimum degree of non-leaves is sufficiently large. Theorem $\ref{thm:othergeneralcat}$ helps begin to resolve the case in which $m = 2^{i} + 1$.  Combined, they allow us to find upper bounds on the saturation number of a large portion of caterpillars.

\subsection{Beyond Caterpillars}

Finally, we demonstrate how perfect dominating sets can be used to find bounds for trees other than caterpillars.  In this section, we give only the proof of Theorem $\ref{thm:vgsup}$; the proof of the more general Theorem $\ref{thm:asymmetricgeneral}$ is very similar.  
\subsubsection{Proof of Theorem $\ref{thm:vgsup}$}
\begin{proof}
As before, we begin by considering a $\kappa$-construction with $i = k-1$, in which we have perfect dominating set $S$ of $Q_{k-1}$'s.  We will build our saturated subgraph $H$ in $Q_{\kappa}$.  To start, add to $H$ all edges within and incident to $S$.  Then, add edges between vertices not in $S$ that are adjacent to the same $Q_{k-1}$ in $S$.  Because $r \ge k$, all vertices not in $S$ have a degree of at maximum $r-1$ at this point.  From here, we use Lemma $\ref{lem:bipartitehall}$ on the set of vertices not in $S$ to consistently add perfect matchings until their degree is exactly $r-1$.  At this point, we claim that $H$ is saturated.

To show this, we first note that $H$ must be free of $VGS_{k, m}$.  We can see this because, first of all, the central node of the very generalized star must be in $S$.  Otherwise, since, by assumption, only one of its adjacent vertices has degree greater than or equal to $r$, we cannot have a copy of $VGS_{k, m}$.  However, even if the central vertex is in a $Q_{k-1}$, only $k-1$ of the legs of a $VGS_{k, m}$ can be placed in this $Q_{k-1}$ for obvious reasons.  The remaining leg must use two vertices outside of $S$ as the first two vertices of its final leg.  However, since these vertices have degree $r-1$ and, by assumption, no leg has both initial vertices with degree $r-1$ or less, $H$ is $VGS_{k,m}$-free.

However, if we add any non-edge, by construction it must connect two vertices of degree $r-1$ adjacent to different $Q_{k-1}$'s.  Since now our construction with a central vertex in $S$ creates a copy of $VGS_{k, m}$ by using the leg $i$ satisfying $\min\{\max\{k_{i1}, k_{i2}\}\}  =r$ as our final leg and then constructing the remainder of the leg in the adjacent $Q_{k-1}$ to the second vertex, $H$ is saturated.  Note that we can choose $\kappa$ accordingly so that all of the degrees of stars on the remaining legs are sufficiently large while still remaining vertex-disjoint.

To scale this up to $Q_{n}$, we simply place copies of $H$ at the vertices of $Q_{\kappa} \ \Box \ Q_{n-\kappa}$ and no edges in between.  To see that this maintains saturation, consider the addition of a non-edge between copies of $H$ in $Q_n$.  There are two cases.  First, this non-edge connects vertices of degree $r-1$. In this case, the same construction as in $Q_{\kappa}$ creates a copy of $VGS_{k, m}$, since these vertices are adjacent to different $Q_{k-1}$'s.  The second case is that the non-edge connects two vertices in $S$.  In this case, $k-1$ of our legs can be placed in one $Q_{k-1}$ and the remaining leg, after traversing the newly added edge, in the other $Q_{k-1}$, thereby creating $VGS_{k, m}$.  Therefore, our overall subgraph of $Q_n$ is saturated, and all that remains is to compute the number of edges in our saturated subgraph, giving us an upper bound on the saturation number of:
\[2^{n-\kappa} \cdot \left(\frac{2^{\kappa - k}}{\kappa - k +2} \cdot \kappa 2^{k} + (r-2) \cdot 2^{\kappa - 1} \right) \sim r \cdot 2^{n-1}. \qedhere \]
\end{proof}

Notice that a similar proof can be used to find upper bounds on unbalanced very generalized stars (where the legs are of different lengths), using the same argument as in the generalized star case.  

%%%%%%%%%%%%%%%%%%%%%%%%%%%%%%%%%%%%%%%%%%%
%%%%%%%%%%%%%%%%%%%%%%%%%%%%%%%%%%%%%%%%%%%

\section{Conclusion}
\label{sec:conclusion}

In this paper, we examined the saturation number of many forbidden graphs in the hypercube.  We first explored general methods for finding lower and upper bounds on this saturation number for both general graphs and general trees.  We continued by examining specific trees, and then used these to deduce upper bounds on the saturation number of sufficiently high-degree caterpillars.  From these bounds, we suggested two major methods for tackling saturation problems for trees: disjoint subcubes and the Hamming code.  In either case, to fully classify all trees, the question of the exact cubical dimension of a given tree will likely have to be answered.  We, regardless, conjecture the following for trees of sufficiently large minimum degree with respect to their diameter:
\begin{conjecture}
\label{conj:awesomeconjecture}
Given a tree $T$ with $\emin(T) = \delta$, $\sat(Q_n, T) \le (\delta + C-1)\cdot 2^{n-1}$ where $C$ is the maximum distance from any given vertex to the longest central path in $T$.
\end{conjecture}

In the future, we hope to further examine the saturation number of trees, perhaps culminating in a complete classification.  From there, we wish to move onto classifying the saturation number of cycles in the hypercube, as first proposed in \cite{Morrison}, hopefully ultimately moving towards a complete classification of the saturation number in the hypercube.

\section{Acknowledgments}
We would like to thank Dr. Tanya Khovanova and Dr. John Rickert for helpful suggestions and review of the paper.  We would also like to thank Professor Pavel Etingof, Professor David Jerison, and Dr. Slava Gerovitch for coordinating the research process.  Finally, we would like to acknowledge the Center for Excellence in Education, the Research Science Institute, and the Massachusetts Institute of Technology for their help and support.

\bibliographystyle{ieeetr}
\bibliography{biblio}

\end{document}